\documentclass[11pt]{article}
\usepackage{a4wide}
\usepackage{amssymb,amsmath}
\usepackage{color}
\newtheorem{theorem}{Theorem}[section]

\newtheorem{hypothesis}[theorem]{Hypothesis}

\newtheorem{definition}[theorem]{Definition}

\newtheorem{experiment}[theorem]{Experiment}

\newcommand {\rank}     {\mathop{\rm rank}\nolimits}

\newcommand {\kernel}   {\mathop{\rm kernel}\nolimits}

\newcommand {\simnew}{\stackrel{\hbox to 0pt{\rm\scriptsize\hss new\hss}}{\sim}}
\newcommand {\simdif}{\stackrel{\hbox to 0pt{\rm\scriptsize\hss dif\hss}}{\sim}}
\newcommand {\ddt}{\vphantom{\tilde E}{\textstyle{d\over dt}}}

\def\.{^{\vphantom{T}}}

\def\w{\dot x,\ldots,x^{(\mu+1)}}
\def\z{t,x,\w}
\def\mat#1{\left[\begin{array}{#1}}
\def\rix{\end{array}\right]}

\def\msize#1#2{{#1\times #2}}

\def\L#1{{\mathbb L}^{#1}\kern-4pt\raise-2pt\hbox{$\scriptstyle{\mu^{#1}}$}}
\def\inner#1{{\mathop{\mathbb #1}\limits^{\kern3pt\raise-3pt\hbox{$\scriptscriptstyle\circ$}}}{}}

\def\Bbb{\mathbb}

\def\deriv#1#2{\langle #1,#2\rangle}
\def\GL{\mathop{\strut\rm GL}\nolimits}
\def\Sp{\mathop{\strut\rm Sp}\nolimits}
\def\O{\mathop{\strut\rm O}\nolimits}
\title{Discretization of inherent ODEs and the geometric integration of DAEs
with symmetries\thanks{Partially supported by the Research In Pairs program of
{\sl Mathematisches Forschungsinstitut Oberwolfach},
whose hospitality is gratefully acknowledged.}
}
\author{
Peter Kunkel\thanks{
Mathematisches Institut, Universit\"at Leipzig, Augustusplatz 10,
D-04109 Leipzig, Fed.\ Rep.\ Germany, \texttt{kunkel@math.uni-leipzig.de}.}
\and
Volker Mehrmann\thanks{
Institut f\"ur Mathematik, TU Berlin, Str.\ des 17.~Juni 136,
D-10623 Berlin, Fed.\ Rep.\ Germany, \texttt{mehrmann@math.tu-berlin.de}.
Partially supported by the {\it Deutsche Forschungsgemeinschaft} through
Project A2 of CRC 910 \emph{Control of self-organizing nonlinear systems: Theoretical methods and concepts of application}.
}}
\begin{document}

\maketitle
%%%%%%%%%%%%%%%%%%%%%%%%%%%%%%%%%%%%%%%%%%%%%%%%%%%%%%%%%%%%%%%%%%%%%%%%
{\bf Abstract.}
Discretization methods for differential-algebraic equations (DAEs) are considered
that are based on the integration of an associated inherent ordinary differential equation (ODE).
This allows to make use of any discretization scheme suitable for the numerical integration of ODEs.
For DAEs with symmetries it is shown that the inherent ODE can be constructed in such a way
that it inherits the symmetry properties of the given DAE and geometric properties of its flow.
This in particular allows the use of geometric integration schemes with a numerical flow
that has analogous geometric properties.


{\bf AMS(MOS) subject classification.} 37J06, 65L80, 65L05, 65P10.
%%%%%%%%%%%%%%%%%%%%%%%%%%%%%%%%%%%%%%%%%%%%%%%%%%%%%%%%%%%%%%%%%%%%%%%%
\section{Introduction}\label{sec:intro}

We consider the numerical solution of general nonlinear systems
of differential-algebraic equations (DAEs)
\begin{equation}\label{nldae}
F(t,x,\dot x)=0,\quad
F\in C({\mathbb I}\times{\mathbb D}_x\times{\mathbb D}_{\dot x},{\mathbb R}^n)\mbox{ sufficiently smooth},
\end{equation}
where ${\mathbb D}_x,{\mathbb D}_{\dot x}\subseteq{\mathbb R}^n$ are open domains
and ${\mathbb I}\subseteq{\mathbb R}$ is a compact non-trivial interval,
together with a given initial condition
\begin{equation}\label{ic}
x(t_0)=x_0,\quad t_0\in{\mathbb I},\ x_0\in{\mathbb D}_x.
\end{equation}
For this task, numerous discretization schemes that work directly on \eqref{nldae}
or on some index-reduced reformulation have been given in the literature, see e.g.\ \cite{HaiLR88,HaiW96,KunM06}.
In this paper, we consider discretization schemes that work on a so-called
\emph{inherent ordinary differential equation} (ODE) of the given DAE.
The advantage of such an approach is that we can make use of
any discretization scheme suitable for the numerical integration of ODEs.
In particular, if we are able to choose the inherent ODE in such a way
that it inherits symmetry properties of the given DAE and thus special properties
of its flow, we may be in the situation to use \emph{geometric integration}, i.e., to use
special discretization schemes whose numerical flow possesses similar geometric properties, see \cite{HaiLW02}.
Especially in the latter case, we will concentrate on linear time-varying DAEs
\begin{equation}\label{lindae}
E(t) \dot x =A(t) x+f(t),\quad
\begin{array}{l}
E,A\in C({\mathbb I},{\mathbb R}^{n,n}),\
f\in C({\mathbb I},{\mathbb R}^{n})\mbox{ sufficiently smooth},
\end{array}
\end{equation}
where we are interested in the following symmetry properties.

\begin{definition}\label{def:self}
The DAE (\ref{lindae}) and its associated pair $(E,A)$ of matrix functions are called
\emph{self-adjoint} if
\begin{equation}\label{self}
E^T=-E,\quad A^T=A+\dot E
\end{equation}
as equality of functions.
\end{definition}

\begin{definition}\label{def:skew}
The DAE (\ref{lindae}) and its associated pair $(E,A)$ of matrix functions are called
\emph{skew-adjoint} if
\begin{equation}\label{skew}
E^T=E,\quad A^T=-A-\dot E
\end{equation}
as equality of functions.
\end{definition}

In the case of linear ODEs
\begin{equation}\label{linode}
\dot x=A(t)x+f(t)
\end{equation}
it is well-studied how symmetry properties of the matrix function~$A$ are transferred to properties
of the \emph{flow} $\Phi\in C^1({\mathbb I},{\mathbb R}^{n,n})$ defined by
\begin{equation}\label{flow}
\dot\Phi=A(t)\Phi,\quad\Phi(t_0)=I_n.
\end{equation}
In the context of geometric integration, one is especially interested in flows that lie
in a \emph{quadratic Lie group}
\begin{equation}\label{group}
{\mathbb G}=\{G\in\GL(n)\mid G^TXG=X\},
\end{equation}
with some given $X\in{\mathbb R}^{n,n}$ and $\GL(n)$ denoting the
general linear group of invertible matrices in ${\mathbb R}^{n,n}$.
In this case, there are then numerical
integration schemes such as \emph{Gau\ss\ collocation} which conserve quadratic invariants
such that their numerical flow lies
in the Lie group as well, see again \cite{HaiLW02}.
Actually, the flow lies in ${\mathbb G}$ when~$A$ lies pointwise
in the associated \emph{Lie algebra}
\begin{equation}\label{algebra}
{\mathbb A}=\{A\in{\mathbb R}^{n,n}\mid A^TX+XA=0\}.
\end{equation}
This can be seen from $\Phi(t_0)^TX\Phi(t_0)=X$ and
\[
\ddt(\Phi^TX\Phi)=\dot\Phi^TX\Phi+\Phi^TX\dot\Phi=
\Phi^TA^TX\Phi+\Phi^TXA\Phi=\Phi^T(A^TX+XA)\Phi=0.
\]
Following \cite{KunM22a},
we are concerned with the quadratic Lie group $\Sp(2p)$
of symplectic matrices related to
\begin{equation}\label{symplectic}
X=J,\quad J=\mat{cc}0&I_p\\-I_p&0\rix
\end{equation}
and the associated Lie algebra of \emph{Hamiltonian matrices} in the case of self-adjoint DAEs
and with the quadratic Lie group $\O(p,q)$ of \emph{generalized orthogonal matrices} related to
\begin{equation}\label{orthogonal}
X=S,\quad S=\mat{cc}I_p&0\\0&-I_q\rix
\end{equation}
in the case of skew-adjoint DAEs.

%The paper is organized as follows.
%%%%%%%%%%%%%%%%%%%%%%%%%%%%%%%%%%%%%%%%%%%%%%%%%%%%%%%%%%%%%%%%%%%%%%%%
\section{Preliminaries}\label{sec:prelim}

In the following, we give a concise overview of the relevant theory on DAEs
that we make use of, see e.g.\ \cite{KunM06}.
The basis are the so-called \emph{derivative array equations}
\begin{equation}\label{infl}
F_\ell(t,x,\dot x,\ldots,x^{(\ell+1)})=0,
\end{equation}
see~\cite{Cam87a}, where $F_\ell$ has the form
\[
F_\ell(t,x,\dot x,\ldots,x^{(\ell+1)}) = \mat{c}
F(t,x,\dot x)\\
{\textstyle{d\over dt}}F(t,x,\dot x)\\
\vdots\\
\big({\textstyle{d\over dt}}\big)^\ell F(t,x,\dot x)
\rix
\]
with Jacobians (denoting the derivative of $F$ with respect to the variable $x$ by $F_x$
and accordingly)
\begin{equation}\label{jacs}
\begin{split}
M_\ell(t,x,\dot x,\dots,x^{(\ell+1)})&=
F_{\ell;\dot x,\dots,x^{(\ell+1)}}(t,x,\dot x,\dots,x^{(\ell+1)}),
\\
N_\ell(t,x,\dot x,\dots,x^{(\ell+1)})&=
-[\>F_{\ell;x}(t,x,\dot x,\dots,x^{(\ell+1)})\>\>0\>\>\ldots\>\>0\>].
\end{split}
\end{equation}
The following hypothesis then states sufficient conditions for the given DAE
to describe a \emph{regular problem}.

\begin{hypothesis}\label{hyp:hyp}
There exist integers $\mu$, $a$, and~$d$ such that the set
\begin{equation}\label{Ch4-set}
 {\Bbb L}_\mu=\{(\z)\in{\Bbb R}^{(\mu+2)n+1}\mid F_{\mu}(\z)=0\}
\end{equation}
associated with~$F$ is nonempty and such that
for every
$(t_0,x_0,\dot x_0,\ldots,x^{(\mu+1)}_0)\hskip-1pt\in\hskip-1pt {\Bbb L}_\mu$,
there exists a $($sufficiently small$)$ neighborhood
in which the following properties hold:
\begin{enumerate}
\item
We have $\rank M_{\mu}(\z)=(\mu+1)n-a$ on~${\Bbb L}_\mu$ such
that there exists a smooth matrix function~$Z_2$
of size $\msize{(\mu+1)n}{a}$ and pointwise maximal rank,
satisfying $Z_2^TM_{\mu}=0$ on ${\Bbb L}_\mu$.
\item
We have $\rank\hat A_2(\z)=a$, where
$\hat A_2=Z_2^TN_{\mu}[I_n\>0\>\cdots\>0]^T$
such that there exists a smooth matrix function~$T_2$
of size $\msize{n}{d}$, $d=n-a$, and pointwise maximal rank,
satisfying $\hat A_2T_2=0$.
\item
We have $\rank F_{\dot x}(t,x,\dot x)T_2(\z)=d$ such
that there exists a smooth matrix function~$Z_1$
of size $\msize{n}{d}$ and pointwise maximal rank,
satisfying $\rank\hat E_1T_2=d$, where $\hat E_1=Z_1^TF_{\dot x}$.
\end{enumerate}
\end{hypothesis}

Note that the local existence of functions $Z_2,T_2,Z_1$ can be guaranteed by the
application of the implicit function theorem, see~\cite[Theorem~4.3]{KunM06}.
Moreover, we may assume that they possess (pointwise) orthonormal columns.
Note also that due to the full rank requirement we may choose $Z_1$ to be constant.

Following the presentation in \cite{Kun15}, we use the shorthand notation
$y=(\dot x,\ldots,x^{(\mu+1)})$ and $\smash{y_0=(\dot x_0,\ldots,x^{(\mu+1)}_0)}$.
The system of nonlinear equations
\begin{equation}\label{implicit}
H(t,x,y,\alpha)=\left[\begin{array}{c}
F_\mu(t,x,y)-Z_{2,0}\alpha\\
T_{1,0}^T(y-y_0)
\end{array}\right],
\end{equation}
with the columns of $T_{1,0}$ forming an orthonormal basis of $\kernel F_{\mu;y}(t_0,x_0,y_0)$
and $Z_{2,0}=Z_2(t_0,x_0,y_0)$ according to Hypothesis~\ref{hyp:hyp},
is then locally solvable for $y,\alpha$ in terms of $(t,x)$ due to the
implicit function theorem. In particular, $\alpha=\hat F_2(t,x)$
with some function~$\hat F_2$. One can show that $\hat F_2(t,x)=0$ describes the
whole set of algebraic constraints implied by the original DAE.
Setting furthermore $\hat F_1(t,x,\dot x)=Z_1^TF(t,x,\dot x)$ yields
a so-called \emph{reduced DAE}
\begin{equation}\label{red1}
\begin{array}{ll}
\hat F_1(t,x,\dot x)=0,\qquad&\mbox{($d$ differential equations)}\\[1mm]
\hat F_2(t,x)=0,&\mbox{($a$ algebraic equations)}
\end{array}
\end{equation}
in the sense that it satisfies Hypothesis~\ref{hyp:hyp} with $\mu=0$.

Moreover, one can show that $\hat F_{2;x}$ possesses full row rank
implying that we can split $x$ possibly after a renumeration of the components
according to $x=(x_1,x_2)$ such that $\hat F_{2;x_2}$ is nonsingular.
The implicit function theorem then yields $x_2={\mathcal R}(t,x_1)$
with some function~${\mathcal R}$. Differentiating this relation to
eliminate $x_2$ and $\dot x_2$ in the first equation of \eqref{red1},
we can apply the implicit function theorem once more (requiring the
solvability of the DAE) yielding $x_1={\mathcal L}(t,x_1)$,
a so-called \emph{inherent ODE}, with some function~${\mathcal L}$.
Putting both parts together, we end up with a second kind of reduced DAE
\begin{equation}\label{red2}
\begin{array}{ll}
x_1={\mathcal L}(t,x_1),\qquad&\mbox{($d$ differential equations)}\\[1mm]
x_2={\mathcal R}(t,x_1).&\mbox{($a$ algebraic equations)}
\end{array}
\end{equation}

Note that, once we have fixed the splitting of the variables, the constructed functions
${\mathcal L}$ and ${\mathcal R}$ are unique. In particular,
the set ${\mathbb L}_{\mu+1}$ can be locally parameterized according to
\begin{equation}\label{par}
F_{\mu+1}(t,x_1,{\mathcal R}(t,x_1),{\mathcal L}(t,x_1),
{\mathcal R}_t(t,x_1)+{\mathcal R}_{x_1}(t,x_1){\mathcal L}(t,x_1),{\mathcal W}(t,x_1,p))\equiv0
\end{equation}
with a suitable parameter $p\in{\mathbb R}^a$ and a related function ${\mathcal W}$.

Under some technical assumptions, see \cite{KunM06}, the original DAE and the reduced DAEs \eqref{red1} and \eqref{red2}
possess the same solutions. As a consequence, we may discretize the
reduced DAEs instead of the original DAE utilizing the better properties of the latter ones.
But this requires the possibility to evaluate the implicitly defined functions.
In the case of $\hat F_2$ in \eqref{red1} the standard approach, see \cite{KunM06}, is to go back to
the definition of $\hat F_2$ in such a way that we replace $\hat F_2(t,x)=0$ by $F_\mu(t,x,y)=0$.

In the special case of linear time-varying DAEs \eqref{lindae},
the Jacobians $M_\mu,N_\mu$ used in Hypothesis~\ref{hyp:hyp} only depend on~$t$
such that the functions $Z_2,T_2,Z_1$ can be chosen to depend also only on~$t$.
The corresponding reduced DAE \eqref{red1} then takes the form
\begin{equation}\label{red1lin}
\arraycolsep 1.2pt
\begin{array}{rll}
\hat E_1(t)\dot x&=\hat A_1(t)x+\hat f_1(t),\qquad&\mbox{($d$ differential equations)}\\[1mm]
0&=\hat A_2(t)x+\hat f_2(t),&\mbox{($a$ algebraic equations)}
\end{array}
\end{equation}
where
\begin{equation}\label{red1lin1}
\begin{array}{lll}
\hat E_1=Z_1^TE,&\hat A_1=Z_1^TA,&\hat f_1=Z_1^Tf,\\
&\hat A_2=Z_2^TN_\mu[I_n\>0\>\cdots\>0]^T,&\hat f_2=Z_2^Tg_\mu
\end{array}
\end{equation}
with
\begin{equation}\label{red1lin2}
M_\mu=\mat{cccc}
E\\
\dot E-A&E\\
\ddot E-2\dot A&2\dot E-A&E\\
\vdots&\vdots&\ddots&\ddots\rix,\
N_\mu=\mat{cccc}
A&0&\cdots&0\\
\dot A&0&\cdots&0\\
\ddot A&0&\cdots&0\\
\vdots&\vdots&&\vdots\rix,\
g_\mu=\mat{c}
f\\
\dot f\\
\ddot f\\
\vdots\rix.
\end{equation}
The splitting of the variables as $x=(x_1,x_2)$ that leads to second form of a reduced DAE
corresponds to a splitting of $\hat A_2=[\>A_{21}\>\>A_{22}\>]$ with the
requirement that $A_{22}$ is pointwise nonsingular.
It is then obvious that we can solve
the second equation of (\ref{red1lin}) for $x_2$ in terms of $x_1$,
differentiate, and eliminate $x_2$ and $\dot x_2$ in the first equation
of (\ref{red1lin}) to obtain a linear version of (\ref{red2}).

In order to utilize global canonical forms as they were presented in \cite{KunM22a},
we observe that the construction of \eqref{red1lin} transforms covariantly
with global equivalence transformations as follows. Let $(\tilde E,\tilde A)$ be globally
equivalent to $(E,A)$, i.e., let sufficiently smooth, pointwise nonsingular
matrix functions $P\in C({\mathbb I},{\mathbb R}^{n,n})$ and
$Q\in C^1({\mathbb I},{\mathbb R}^{n,n})$ be given such that
\begin{equation}\label{equiv}
\tilde E=PEQ,\quad\tilde A=PAQ-PE\dot Q,
\end{equation}
describing scalings of the DAE \eqref{lindae} and the unknown~$x$, respectively.
The corresponding Jacobians are then related by
\begin{equation}\label{jacrel1}
\tilde M_\mu=\Pi_\mu M_\mu\Theta_\mu,\quad\tilde N_\mu=\Pi_\mu N_\mu\Theta_\mu-\Pi_\mu M_\mu\Psi_\mu
\end{equation}
with
\begin{equation}\label{jacrel2}
\Pi_\mu=\mat{cccc}
P\\
\dot P&P\\
\ddot P&2\dot P&P\\
\vdots&\vdots&\ddots&\ddots\rix,\
\Theta_\mu=\mat{cccc}
Q\\
2\dot Q&Q\\
3\ddot Q&3\dot Q&Q\\
\vdots&\vdots&\ddots&\ddots\rix,\
\Psi_\mu=\mat{cccc}
\dot Q&0&\cdots&0\\
\ddot Q&0&\cdots&0\\
\dddot Q&0&\cdots&0\\
\vdots&\vdots&&\vdots\rix.
\end{equation}
With given choices $Z_2,T_2,Z_1$ for $(E,A)$ along Hypothesis~\ref{hyp:hyp}
we may choose $\tilde Z_2,\tilde T_2,\tilde Z_1$ for $(\tilde E,\tilde A)$ as
\begin{equation}\label{redrel}
\tilde Z_2^T=Z_2^T\Pi_\mu^{-1},\quad
\tilde T_2=Q^{-1}T_2,\quad
\tilde Z_1^T=Z_1^TP^{-1}.
\end{equation}

Having summarized the theory for general nonlinear and linear time-varying DAEs,
the next section deals with the construction of a suitable inherent ODEs
for a given DAE.
%%%%%%%%%%%%%%%%%%%%%%%%%%%%%%%%%%%%%%%%%%%%%%%%%%%%%%%%%%%%%%%%%%%%%%%%
\section{Construction and Evaluation of an Inherent ODE}\label{sec:inherent}

To get more flexibility into the choice of an inherent ODE,
we introduce a (linear but in general time-dependent)
transformation of the unknown~$x$ before we perform the splitting,
i.e., we consider
\begin{equation}\label{xtrf}
x=Q(t)\mat{c}x_1\\x_2\rix,
\end{equation}
where $Q\in C^1({\mathbb I},{\mathbb R}^{n,n})$ is sufficiently smooth
and pointwise nonsingular.
According to \cite[Lemma 4.6]{KunM06} the so transformed DAE \eqref{nldae}
satisfies Hypothesis~\ref{hyp:hyp} as well with the same characteristic values $\mu,a,d$.
As before, the only requirement for~$Q$ is that we can solve the algebraic constraints
for $x_2$ in terms of~$x_1$. Writing
\begin{equation}\label{Qsplit}
Q=[\>T_2\.\>\>T_2'\>],
\end{equation}
the algebraic constraints read
\[
\hat F_2(t,T_2\.x_1+T_2'x_2)=0.
\]
Hence, in order to be able to solve for~$x_2$ we need
$\hat F_{2;x}T_2'$ to be pointwise nonsingular.
If this is the case, then the chosen $Q$ fixes a reduced DAE
of the form \eqref{red2} satisfying
\begin{equation}\label{partrf}
\begin{array}{l}
F_{\mu+1}\left(t,Q(t)\mat{c}x_1\\x_2\rix,
Q(t)\mat{c}\dot x_1\\\dot x_2\rix+
\dot Q(t)\mat{c}x_1\\x_2\rix,{\mathcal W}(t,x_1,p)\right)\equiv0,\\[4mm]
\quad \dot x_1={\mathcal L}(t,x_1),\ x_2={\mathcal R}(t,x_1),\
\dot x_2={\mathcal R}_t(t,x_1)+{\mathcal R}_{x_1}(t,x_1){\mathcal L}(t,x_1)
\end{array}
\end{equation}
with a suitable parameter $p\in{\mathbb R}^a$ and a related function ${\mathcal W}$.

For a numerical realization, we are confronted with two problems.
First, we must be able to evaluate the implicitly defined functions ${\mathcal L}$
and ${\mathcal R}$.
Second, for a nontrivial choice of $Q$ we must have access to $\dot Q$.

In the next subsections, we discuss how to overcome these problems.

\subsection{Numerical Evaluation of the Inherent ODE}\label{sub:eval}

The first problem can be dealt with by solving the system of (nonlinear) equations
\begin{equation}\label{red2eval}
F_{\mu+1}(t,x,\dot x,w)=0,\quad[\>I_d\>\>0\>]Q(t)^{-1}x=x_1
\end{equation}
for given $(t,x_1)$. Because of the first part in \eqref{red2eval},
at a solution, the resulting $(t,x,\dot x,w)$ must satisfy
\[
x=Q(t)\mat{c}x_1\\{\mathcal R}(t,x_1)\rix,\
\dot x=Q(t)\mat{c}{\mathcal L}(t,x_1)\\{\mathcal R}_t(t,x_1)+{\mathcal R}_{x_1}(t,x_1){\mathcal L}(t,x_1)\rix+\dot Q(t)\mat{c}x_1\\{\mathcal R}(t,x_1)\rix.
\]
Because of the second part in \eqref{red2eval}, we regain the prescribed $x_1$.
Furthermore, we observe that
\[
{\mathcal R}(t,x_1)=[\>0\>\>I_a\>]Q(t)^{-1}x,\
{\mathcal L}(t,x_1)=[\>I_d\>\>0\>]Q(t)^{-1}(\dot x-\dot Q(t)Q(t)^{-1}x)
\]
yielding the required evaluations of ${\mathcal L}$ and ${\mathcal R}$.

Since \eqref{red2eval} constitutes an underdetermined system of equations,
the method of choice to solve \eqref{red2eval} numerically is the \emph{Gau\ss-Newton method}.
In order to show that the Gau\ss-Newton method will convergence quadratically
for sufficiently good starting values, we need to show that the Jacobian at a solution
possesses full row rank, see e.g.~\cite{Deu04}.

\begin{theorem}\label{th:quad}
Let \eqref{nldae} satisfy Hypothesis~\ref{hyp:hyp} both with $\mu,a,d$ and with $\mu+1,a,d$.
Then, the Jacobian of \eqref{red2eval} possesses full row rank at every solution
provided that $\smash{\hat F_{2;x}\.T_2'}$ is pointwise nonsingular.
\end{theorem}
\begin{proof}
Due to \eqref{implicit} for $\mu+1$ replacing $\mu$ we have
\[
F_{\mu+1;x}-Z_{2,0}\hat F_{2;x}=0,
\]
omitting for convenience the arguments here and later. Hence,
\[
\hat F_{2;x}=(Z_2^TZ_{2,0}\.)^{-1}Z_2^TF_{\mu+1;x}
\]
in a sufficiently small neighborhood.
Completing $Z_2$ to a pointwise nonsingular matrix function $[\>Z_2'\>\>Z_2\.\>]$,
elementary row operations of the Jacobian of the first part in \eqref{red2eval} yield
\[
\mat{cc}F_{\mu+1;x}&F_{\mu+1;\dot x,\ldots,x^{(\mu+2)}}\rix\rightarrow
\mat{cc}Z_2'^{T}F_{\mu+1;x}&Z_2'^{T}F_{\mu+1;\dot x,\ldots,x^{(\mu+2)}}\\
Z_2^TF_{\mu+1;x}&0\rix.
\]
According to Hypothesis~\ref{hyp:hyp} the entry $\smash{Z_2'^{T}F_{\mu+1;\dot x,\ldots,x^{(\mu+2)}}}$
possesses full row rank such that we are left with the entry $Z_2^TF_{\mu+1;x}$ together
with the Jacobian $[\>I_d\>\>0\>]Q^{-1}$ of the second equation in \eqref{red2eval}.
Multiplying the first part with $(Z_2^TZ_{2,0}\.)^{-1}$ from the left and both parts
with $Q$ from the right yields the matrix function
\[
\mat{cc}\hat F_{2;x}T_2\.&\hat F_{2;x}T_2'\\I_d&0\rix
\]
which is pointwise nonsingular provided that $\hat F_{2;x}T_2'$ is pointwise nonsingular.
\end{proof}

\subsection{Numerical Construction of the Transformation}\label{sub:trans}

It remains the question how we can deal with $\dot Q$ in extracting
the evaluation of ${\mathcal L}(t,x_1)$.
In particular, we are interested in applications where a trivial choice
as constant $Q$ or beforehand given $Q$ with implemented functions
to evaluate both $Q(t)$ and $\dot Q(t)$ is not possible but where
$Q$ has to be chosen numerically during the integration of the DAE.
The main problem in this context is that we must choose $Q$ in a smooth way,
at least on the current interval $[t_0,t_0+h]$ of the numerical integration
with $h>0$ sufficiently small, and that we must be able to evaluate~$\dot Q$.

The approach we will follow here is automatic differentiation, see \cite{Gri89}.
This means that we work not only with the value of a variable
but with a pair of numbers that represent the value and the derivative of a variable.
Operations on such pairs are then defined by means of the known differentiation rules.
If we use the notation $\deriv{x}{\dot x}$ for such a pair, the typical operations
used in linear algebra then read
\begin{equation}\label{oper}
\begin{array}{ll}
\mbox{\rm(a)\qquad}
\deriv{x}{\dot x}+\deriv{y}{\dot y}=\deriv{x+y}{\dot x+\dot y},\\[1mm]
\mbox{\rm(b)\qquad}
\deriv{x}{\dot x}-\deriv{y}{\dot y}=\deriv{x-y}{\dot x-\dot y},\\[1mm]
\mbox{\rm(c)\qquad}
\deriv{x}{\dot x}\cdot\deriv{y}{\dot y}=\deriv{x\cdot y}{\dot x\cdot y+x\cdot\dot y},\\[1mm]
\mbox{\rm(d)\qquad}
\deriv{x}{\dot x}/\deriv{y}{\dot y}=\deriv{x/y}{(\dot x-x\cdot\dot y/y)/y},\\[1mm]
\mbox{\rm(e)\qquad}
\sqrt{\deriv{x}{\dot x}}=\deriv{\sqrt x}{\frac12\dot x/\sqrt x}.
\end{array}
\end{equation}
These operations can be obviously extended in a componentwise way to vector and matrix operations.

Note that in a programming language like \texttt{C++} this approach can be implemented
by defining a corresponding new class and overloading the above operations
to work with this class. In this way it is possible to perform tasks
of linear algebra like Cholesky decomposition $A=L\cdot L^T$ in a smooth way
yielding $\deriv{L}{\dot L}$ for given $\deriv{A}{\dot A}$. This is valid
for all numerical algorithms that do not include if-clauses.
If there are if-clauses,
as for example in the QR decomposition $A\cdot\Pi=Q\cdot R$, then
we can at least locally get a smooth version. To do this for
the QR decomposition,
we may proceed as follows. For a reference point, typically $t_0$,
we perform a standard QR decomposition $A(t_0)\cdot\Pi_0=Q_0\cdot R_0$.
We then freeze all if-clauses and use automatic differentiation in
the evaluation of the QR decomposition $A\cdot\Pi_0=Q\cdot R$.
In this way, we get $\deriv{Q}{\dot Q}$ and $\deriv{R}{\dot R}$
for given $\deriv{A}{\dot A}$.

In particular, we can use this approach to perform the construction of reduced DAEs for
linear time-varying systems as described in Section~\ref{sec:prelim}
with the aim to get not only values for the involved transformations
but also values for their derivatives.

To start the construction of the reduced system \eqref{red1lin}, we need $\dot M_\mu,\dot N_\mu,\dot g_\mu$
besides $M_\mu,N_\mu,g_\mu$. Writing $M,N,g$ for the formally infinite
extensions of $M_\mu,N_\mu,g_\mu$ and defining
\[
S=\mat{cccc}0\\I_n&0\\&I_n&0\\&&\ddots&\ddots\rix,\quad
V=\mat{c}I_n\\0\\0\\\vdots\rix
\]
we have the relations
\[
\dot M=S^TM-MS^T+N,\quad\dot N=S^TN,\quad\dot g=S^Tg,
\]
see \cite{CamK09}. Hence, from the evaluations $M_{\mu+1},N_{\mu+1},g_{\mu+1}$
we can actually retrieve the desired
$\deriv{M_\mu}{\dot M_\mu},\deriv{N_\mu}{\dot N_\mu},\deriv{g_\mu}{\dot g_\mu}$.
A first locally smooth QR decomposition then yields $\deriv{Z_2}{\dot Z_2}$
and thus $\smash{\deriv{\hat A_2}{\ddt{\hat A}_2}}$.
A second locally smooth QR decomposition then gives $\deriv{T_2}{\dot T_2}$
and with a third locally smooth QR decomposition for $\deriv{E}{\dot E}\cdot\deriv{T_2}{\dot T_2}$
we finally get $\deriv{Z_1}{\dot Z_1}$. In the latter case we can also use a
standard QR decomposition once at $t_0$
and use the so obtained $Z_{1,0}$ to set $\deriv{Z_1}{\dot Z_1}=\deriv{Z_{1,0}}{0}$
if it seems more suited.
The remaining quantities of the reduced DAE are then given by automatic differentiation
along the lines of~\eqref{red1lin1}.

With a given choice $\deriv{Q}{\dot Q}$ for fixing an inherent ODE,
transforming the reduced DAE \eqref{red1} by means of \eqref{xtrf} yields
\[
\arraycolsep 1.2pt
\begin{array}{rl}
\hat E_{11}(t)\dot x_1+\hat E_{12}(t)\dot x_2&=\hat A_{11}(t)x_1+\hat A_{12}(t)x_2+\hat f_1(t),\\[1mm]
0&=\hat A_{21}(t)x_1+\hat A_{22}(t)x_2+\hat f_2(t),
\end{array}
\]
where
\[
\begin{array}{ll}
\hat E_{11}=\hat E_1T_2,&\hat E_{12}=\hat E_1T_2',\\
\hat A_{11}=\hat A_1T_2-\hat E_1\dot T_2,&\hat A_{12}=\hat A_1T_2'-\hat E_1\dot T_2',\\
\hat A_{21}=\hat A_2T_2,&\hat A_{22}=\hat A_2T_2',
\end{array}
\]
and we are in the same situation as in the special case described in Section~\ref{sec:prelim}.
In particular, we can solve for $x_2$, differentiate, eliminate, and solve for $\dot x_1$
to get the fixed inherent ODE.

A special choice of $Q$ can be obtained by a locally smooth QR decomposition
of $\smash{\deriv{\hat E_1^T}{\ddt{\hat E}_1^T}}$ leading to $\hat E_{12}=0$.
Hypothesis~\ref{hyp:hyp} then guarantees that $\hat A_{22}$ is pointwise nonsingular.
If we set $Q_0=Q(t_0)$ and $\dot Q_0=\dot Q(t_0)$, we may also replace $Q$
by the constant version $Q(t)=Q_0$ or
by the linearized version $Q(t)=Q_0+(t-t_0)\dot Q_0$. The latter corresponds
to the construction of so-called \emph{spin-stabilized integrators} introduced in \cite{KunM07}.
In the case that $\mu=0$, the constructions can be simplified by using $E$ instead of $\hat E_1$
since no construction of a reduced system is required.

%%%%%%%%%%%%%%%%%%%%%%%%%%%%%%%%%%%%%%%%%%%%%%%%%%%%%%%%%%%%%%%%%%%%%%%%
\section{Symmetries and Geometric Integration}\label{sec:symmetries}

In this section we treat linear time-varying DAEs that are self-adjoint
or skew-adjoint. The aim is to utilize the symmetry in the construction
of a suitable inherent ODE such that it inherits certain properties of the original DAE.
Note that self-adjointness and skew-adjointness are invariant under so-called
congruence, i.e., under global equivalence \eqref{equiv} with $P=Q^T$, see e.g.~\cite{KunM22a}.
As there, we will write $(\tilde E,\tilde A)\equiv(E,A)$ to indicate that the pairs are congruent.
Note also that regularity of a pair $(E,A)$ of sufficiently smooth matrix function
$E,A\in C({\mathbb I},{\mathbb R}^{n,n})$ is necessary and sufficient for the
asscociated DAE \eqref{lindae} to satisfy Hypothesis~\ref{hyp:hyp}, see e.g.~\cite{KunM06}.

\subsection{Self-Adjoint DAEs}\label{sub:self}
Assuming \eqref{self} for \eqref{lindae}, we will make use of the following
global canonical form taken from \cite{KunM22a} in a slightly rephrased version.

\begin{theorem}\label{th:gcfself}
Let $(E,A)$ with $E,A\in C({\mathbb I},{\mathbb R}^{n,n})$ be sufficiently smooth
and let the associated DAE \eqref{lindae} satisfy Hypothesis~\ref{hyp:hyp}.
If $(E,A)$ is self-adjoint, then we have that
\begin{equation}\label{gcfself}
(E,A)\equiv\left(
\left[\begin{array}{ccc} 0&I_{p}&0\\-I_{p}&0&0\\0&0&E_{33} \end{array}\right],
\left[\begin{array}{ccc} 0&0&0\\0&A_{22}&A_{23}\\0&A_{32}&A_{33} \end{array}\right]\right),
\end{equation}
where
\begin{equation}\label{gcfselfq}
E_{33}(t)\dot x_3=A_{33}(t)x_3+f_3(t),
\end{equation}
is uniquely solvable for every sufficiently smooth~$f_3$
without specifying initial conditions. Furthermore,
\begin{equation}\label{gcfselfp}
E_{33}^T=-E_{33}\.,\quad A_{22}^T=A_{22}\.,\quad A_{32}^T=A_{23}\.,\quad A_{33}^T=A_{33}+\dot E_{33}.
\end{equation}
\end{theorem}

In order to construct a suitable reduced DAE \eqref{red1lin}, we follow the lines
of Hypothesis~\ref{hyp:hyp} for the global canonical form, indicated by tildes, and start with
\[
\tilde M_\mu=\mat{ccc|ccc|ccc|c}
0&I_{p}&0&&&&&&&\\
-I_{p}&0&0&&&&&&&\\
0&0&E_{33}&&&&&&&\\\hline
0&0&0&0&I_{p}&0&&&&\\
0&-A_{22}&-A_{23}&-I_{p}&0&0&&&&\\
0&-A_{32}&\dot E_{33}-A_{33}&0&0&E_{33}&&&&\\\hline
0&0&0&0&0&0&0&I_{p}&0\\
0&-2\dot A_{22}&-2\dot A_{23}&0&-A_{22}&-A_{23}&-I_{p}&0&0\\
0&-2\dot A_{32}&\ddot E_{33}-2\dot A_{33}&0&-A_{32}&2\dot E_{33}-A_{33}&0&0&E_{33}\\\hline
\vdots&\vdots&\vdots&\vdots&\vdots&\vdots&\vdots&\vdots&\vdots&\ddots
\rix.
\]
Due to the identities, the only possible rank-deficiency is related to the part
belonging to the pair $(E_{33},A_{33})$. The properties of \eqref{gcfselfq}
then imply that $d=2p$ and $a=n-2p$ in Hypothesis~\ref{hyp:hyp}.
Furthermore, the left null space of~$\tilde M_\mu$ is described by
\[
\tilde Z_2^T=\mat{ccc|ccc|ccc|c}
*&0&\tilde Z_{2,0}^T&*&0&\tilde Z_{2,1}^T&*&0&\tilde Z_{2,2}^T&\cdots
\rix.
\]
Observing that
\[
\tilde N_\mu[I_n\>0\>\cdots\>0]^T=\mat{ccc}
0&0&0\\
0&A_{22}&A_{23}\\
0&A_{32}&A_{33}\\\hline
0&0&0\\
0&\dot A_{22}&\dot A_{23}\\
0&\dot A_{32}&\dot A_{33}\\\hline
0&0&0\\
0&\ddot A_{22}&\ddot A_{23}\\
0&\ddot A_{32}&\ddot A_{33}\\\hline
\vdots&\vdots&\vdots
\rix,
\]
we get
\[
\hat A_2=\mat{ccc}0&\hat A_{32}&I_a\rix
\]
for the second part of Hypothesis~\ref{hyp:hyp}, where the identity
comes from a special choice of $\tilde Z_2^T$. Choosing
\[
\tilde T_2=\mat{cc}I_p&0\\0&I_p\\0&-\hat A_{32}\rix
\]
and $\tilde Z_1=\tilde T_2$ yields
\[
\tilde Z_1^T\tilde E\tilde T_2\.=
\mat{ccc}I_p&0&0\\0&I_p&-\hat A_{32}^T\rix
\mat{ccc}0&I_{p}&0\\-I_{p}&0&0\\0&0&E_{33}\rix
\mat{cc}I_p&0\\0&I_p\\0&-\hat A_{32}\rix=
\mat{cc}0&I_{p}\\-I_{p}&\hat A_{32}^TE_{33}\.\hat A_{32}\.\rix,
\]
which is indeed pointwise nonsingular, thus satisfying the third part of Hypothesis~\ref{hyp:hyp}.
In particular, the special choice $\tilde Z_1=\tilde T_2$ is possible.
According to \eqref{redrel} with $P=Q^T$ we can also choose $Z_1=T_2$
for the original pair such that the reduced DAE inherits some symmetry properties
of the original DAE. Note also that we may assume that $T_2$ possesses pointwise
orthonormal columns.

By construction, the matrix function $T_2^TET_2\.$ is not only pointwise skew-symmetric
but also pointwise nonsingular. We can then proceed similar to~\cite{KunMS14}.
Setting
\[
T_2^TET_2\.=\mat{cc}\bar E&c\\-c^T&0\rix,
\]
there exists a smooth pointwise orthogonal transformation~$U$ with $U^Tc=\alpha e_1$, $\alpha\ne0$,
where $e_1$ denotes the first canonical basis vector of appropriate size, see e.g.~\cite[Theorem 3.9]{KunM06}.
It follows that
\[
\mat{cc}U\\&1\rix^T\!
\mat{cc}\bar E&c\\-c^T&0\rix
\mat{cc}U\\&1\rix=
\mat{cc}U^T\bar EU&\alpha e_1\\-\alpha e^T&0\rix=
\mat{ccc}*&*&\alpha\\{}*&\bar{\bar E}&0\\-\alpha&0&0\rix,
\]
where $\bar{\bar E}$ is again skew-symmetric and pointwise nonsingular.
Thus, inductively after~$p$ steps, we arrive at
\[
W_1^TT_2^TET_2\.W_1\.=\mat{cc}\tilde E_{11}&\tilde E_{12}\\-\tilde E_{12}^T&0\rix,
\]
where $W_1$ collects all the applied transformations.
By construction, $\tilde E_{11}$ is skew-symmetric and $\tilde E_{12}$
is anti-triangular and pointwise nonsingular.
Finally, setting
\[
W_2=\mat{cc}I_p&0\\-\frac12\tilde E_{12}^{-1}\tilde E_{11}\.&\tilde E_{12}^{-1}\rix
\]
yields
\[
W_2^TW_1^TT_2^TET_2\.W_1\.W_2\.=\mat{cc}0&I_p\\-I_p&0\rix=J.
\]

For convenience, we write again $T_2$ instead of the transformed $T_2W_1W_2$.
Completing $T_2$ to a pointwise nonsingular~$Q$
according to \eqref{Qsplit}, we get
\[
Q^TEQ=\mat{cc}J&\hat E_{12}\\{}*&*\rix,\quad
Q^TAQ-Q^TE\dot Q=\mat{cc}C&\hat A_{12}\\{}*&*\rix.
\]
Since self-adjointness is invariant under congruence and $J$ is constant,
the matrix function $C$ is pointwise symmetric.
With \eqref{xtrf} the reduced DAE transforms to
\[
\arraycolsep 1.2pt
\begin{array}{rl}
J\dot x_1+\hat E_{12}(t)\dot x_2&=C(t)x_1+\hat A_{12}(t)x_2+T_2(t)^Tf(t),\\[1mm]
0&=\hat A_{22}(t)x_2+\hat f_2(t),
\end{array}
\]
where $\hat A_{22}\.=\hat A_2\.T_2'$ is pointwise nonsingular.
Solving the second equation for~$x_2$,
differentiating, and eliminating $x_2$ and $\dot x_2$ from the first equation
yields the inherent ODE
\begin{equation}\label{inhself}
\dot x_1=J^{-1}C(t)x_1+\tilde f_1(t)
\end{equation}
with some transformed inhomogeneity~$\tilde f_1$.

\begin{theorem}\label{th:flowself}
Let $(E,A)$ with $E,A\in C({\mathbb I},{\mathbb R}^{n,n})$ be sufficiently smooth
and let the associated DAE \eqref{lindae} satisfy Hypothesis~\ref{hyp:hyp}.
If $(E,A)$ is self-adjoint, then~$Q$ in \eqref{xtrf} can be chosen
from a restricted class of transformations
in such a way that the so constructed inherent ODE possesses a symplectic flow.
\end{theorem}
\begin{proof}
The above construction shows that it is possible to fix an inherent ODE
with a symplectic flow. It is special in the sense that it
works with pointwise orthogonal transformations
with the exception of $W_2$ which transforms within one half of the variables
and adapts the other half to obtain the matrix~$J$ and thus a set of variables
for which the inherent ODE is Hamiltonian.
\end{proof}

In the special case $\mu=0$ a slightly simplified construction is possible.
Here, Hypothesis~\ref{hyp:hyp} says that $E$ has constant rank allowing to
choose $Q$ in the form \eqref{Qsplit} such that
\[
Q^TEQ=\mat{cc}\hat E_{11}&0\\0&0\rix
\]
with $\hat E_{11}\.=T_2^TET_2\.$ pointwise nonsingular.
Then, the same modifications of $T_2$ as before are possible
leading to a modified $T_2$ with $\hat E_{11}=J$.
With the corresponding modified $Q$, observing $ET_2'=0$, we get that
\[
Q^TEQ=\mat{cc}J&0\\0&0\rix,\quad
Q^TAQ-Q^TE\dot Q=\mat{cc}\hat A_{11}&\hat A_{12}\\\hat A_{21}&\hat A_{22}\rix.
\]
Since congruence conserves self-adjointness, see e.g. \cite{KunMS14}, we have
$\hat A_{11}^T=\hat A_{11}\.$, $\hat A_{12}^T=\hat A_{21}$, and $\hat A_{22}^T=\hat A_{22}\.$.
Moreover, Hypothesis~\ref{hyp:hyp} with $\mu=0$ requires that $\hat A_{22}$ is pointwise nonsingular.
The corresponding reduced DAE, which is here just the original DAE,
transforms to
\[
\arraycolsep 1.2pt
\begin{array}{rl}
J\dot x_1&=\hat A_{11}(t)x_1+\hat A_{12}(t)x_2+T_2(t)^Tf(t),\\[1mm]
0&=\hat A_{12}(t)^Tx_1+\hat A_{22}(t)x_2+T_2'(t)^Tf(t).
\end{array}
\]
Solving the second equation for $x_2$ and eliminating it from the first equation,
we again obtain an inherent ODE of the form \eqref{inhself}, where
\[
C=\hat A_{11}\.-\hat A_{12}\.\hat A_{22}^{-1}\hat A_{12}^T
\]
is pointwise symmetric.

Theoretically, all constructions can be performed globally.
For a numerical realization one typically uses locally smooth variants
as described in Section~\ref{sec:inherent}, which in this case is straightforward
on the basis of locally smooth QR decompositions.

\subsection{Skew-Adjoint DAEs}\label{sub:skew}

Assuming \eqref{skew} for \eqref{lindae}, we will make use of the following
global canonical form taken from \cite{KunM22a} in a slightly rephrased version.

\begin{theorem}\label{th:gcfskew}
Let $(E,A)$ with $E,A\in C({\mathbb I},{\mathbb R}^{n,n})$ be sufficiently smooth
and let the associated DAE \eqref{lindae} satisfy Hypothesis~\ref{hyp:hyp}.
If $(E,A)$ is skew-adjoint, then we have that
\begin{equation}\label{gcfskew}
(E,A)\equiv\left(
\left[\begin{array}{ccc} I_{p}&0&0\\0&-I_{q}&0\\0&0&E_{33} \end{array}\right],
\left[\begin{array}{ccc} 0&0&0\\0&0&0\\0&0&A_{33} \end{array}\right]\right),
\end{equation}
where
\begin{equation}\label{gcfskewq}
E_{33}(t)\dot x_3=A_{33}(t)x_3+f_3(t)
\end{equation}
is uniquely solvable for every sufficiently smooth~$f_3$
without specifying initial conditions. Furthermore,
\begin{equation}\label{gcfskewp}
E_{33}^T=E_{33}\.,\quad A_{33}^T=-A_{33}-\dot E_{33}\.
\end{equation}
\end{theorem}

In order to construct a suitable reduced DAE \eqref{red1lin}, we proceed
as in the self-adjoint case using the same notation.
For the canonical form, we have
\[
\tilde M_\mu=\mat{ccc|ccc|ccc|c}
I_{p}&0&0&&&&&&&\\
0&-I_{q}&0&&&&&&&\\
0&0&E_{33}&&&&&&&\\\hline
0&0&0&I_{p}&0&&&&\\
0&0&0&0&-I_{q}&0&&&&\\
0&0&\dot E_{33}-A_{33}&0&0&E_{33}&&&&\\\hline
0&0&0&0&0&0&I_{p}&0&0\\
0&0&0&0&0&0&0&I_{q}&0&0\\
0&0&\ddot E_{33}-2\dot A_{33}&0&0&2\dot E_{33}-A_{33}&0&0&E_{33}\\\hline
\vdots&\vdots&\vdots&\vdots&\vdots&\vdots&\vdots&\vdots&\vdots&\ddots
\rix.
\]
Due to the identities, the only possible rank-deficiency is related to the part
belonging to the pair $(E_{33},A_{33})$. The properties of \eqref{gcfskewq}
then imply that $d=p+q$ and $a=n-(p+q)$ in Hypothesis~\ref{hyp:hyp}.
Furthermore, the left null space of~$\tilde M_\mu$ is described by
\[
\tilde Z_2^T=\mat{ccc|ccc|ccc|c}
0&0&\tilde Z_{2,0}^T&0&0&\tilde Z_{2,1}^T&0&0&\tilde Z_{2,2}^T&\cdots
\rix.
\]
Observing that
\[
\tilde N_\mu[I_n\>0\>\cdots\>0]^T=\mat{ccc}
0&0&0\\
0&0&0\\
0&0&A_{33}\\\hline
0&0&0\\
0&0&0\\
0&0&\dot A_{33}\\\hline
0&0&0\\
0&0&0\\
0&0&\ddot A_{33}\\\hline
\vdots&\vdots&\vdots
\rix,
\]
we get that
\[
\hat A_2=\mat{ccc}0&0&I_a\rix
\]
for the second part of Hypothesis~\ref{hyp:hyp}, where the identity
comes from a special choice of $\tilde Z_2^T$. Choosing
\[
\tilde T_2=\mat{cc}I_p&0\\0&I_q\\0&0\rix
\]
and $\tilde Z_1=\tilde T_2$ yields
\[
\tilde Z_1^T\tilde E\tilde T_2\.=
\mat{ccc}I_p&0&0\\0&I_p&0\rix
\mat{ccc}I_{p}&0&0\\0&-I_{q}&0\\0&0&E_{33}\rix
\mat{cc}I_p&0\\0&I_q\\0&0\rix=
\mat{cc}I_{p}&0\\0&-I_{q}\rix,
\]
which is indeed pointwise nonsingular, thus satisfying the third part of Hypothesis~\ref{hyp:hyp}.
In particular, the special choice $\tilde Z_1=\tilde T_2$ is possible.
According to \eqref{redrel} with $P=Q^T$ we can also choose $Z_1=T_2$
for the original pair such that the reduced DAE inherits some symmetry properties
of the original DAE. Note also that we may assume that $T_2$ possesses pointwise
orthonormal columns.

By construction, the matrix function $T_2^TET_2\.$ is not only pointwise symmetric
but also pointwise nonsingular. We can then apply the results of \cite{Kun20},
which guarantee the existence of a smooth matrix function~$W$ with
\[
W^TT_2^TET_2\.W=\mat{cc}I_p&0\\0&-I_q\rix=S.
\]

For convenience, we write again $T_2$ instead of the transformed $T_2W$.
Completing $T_2$ to a pointwise nonsingular~$Q$
according to \eqref{Qsplit}, we get
\[
Q^TEQ=\mat{cc}S&\hat E_{12}\\{}*&*\rix,\quad
Q^TAQ-Q^TE\dot Q=\mat{cc}J&\hat A_{12}\\{}*&*\rix.
\]
Since skew-adjointness is invariant under congruence, see \cite{BeaMXZ18,KunM22a}, and $S$ is constant,
the matrix function $J$ is pointwise skew-symmetric.
With \eqref{xtrf} the reduced DAE transforms to
\[
\arraycolsep 1.2pt
\begin{array}{rl}
S\dot x_1+\hat E_{12}(t)\dot x_2&=J(t)x_1+\hat A_{12}(t)x_2+T_2(t)^Tf(t),\\[1mm]
0&=\hat A_{22}(t)x_2+\hat f_2(t),
\end{array}
\]
where $\hat A_{22}\.=\hat A_2\.T_2'$ is pointwise nonsingular.
Solving the second equation for~$x_2$,
differentiating, and eliminating $x_2$ and $\dot x_2$ from the first equation
yields the inherent ODE
\begin{equation}\label{inhskew}
\dot x_1=S^{-1}J(t)x_1+\tilde f_1(t)
\end{equation}
with a transformed inhomogeneity~$\tilde f_1$.

\begin{theorem}\label{th:flowskew}
Let $(E,A)$ with $E,A\in C({\mathbb I},{\mathbb R}^{n,n})$ be sufficiently smooth
and let the associated DAE \eqref{lindae} satisfy Hypothesis~\ref{hyp:hyp}.
If $(E,A)$ is skew-adjoint, then~$Q$ in \eqref{xtrf} can be chosen
from a restricted class of transformations
in such a way that the so constructed inherent ODE possesses a generalized orthogonal flow.
\end{theorem}
\begin{proof}
The above construction shows that it is possible to fix an inherent ODE
with a generalized orthogonal flow. It is special in the sense that it
works with pointwise orthogonal transformations
with the exception of $W$.
\end{proof}

In the special case that $\mu=0$, a slightly simplified construction is possible.
Here, Hypothesis~\ref{hyp:hyp} implies that $E$ has constant rank allowing to
choose $Q$ in the form \eqref{Qsplit} such that
\[
Q^TEQ=\mat{cc}\hat E_{11}&0\\0&0\rix
\]
with $\hat E_{11}\.=T_2^TET_2\.$ pointwise nonsingular.
Then, the same modifications of $T_2$ as before are possible
leading to a modified $T_2$ with $\hat E_{11}=S$.
With the corresponding modified $Q$, observing $ET_2'=0$, we get
\[
Q^TEQ=\mat{cc}S&0\\0&0\rix,\quad
Q^TAQ-Q^TE\dot Q=\mat{cc}\hat A_{11}&\hat A_{12}\\\hat A_{21}&\hat A_{22}\rix.
\]
Since congruence transformations conserve skew-adjointness, we have
$\hat A_{11}^T=-\hat A_{11}\.$, $\hat A_{12}^T=-\hat A_{21}$, and $\hat A_{22}^T=-\hat A_{22}\.$.
Moreover, Hypothesis~\ref{hyp:hyp} with $\mu=0$ requires that $\hat A_{22}$ is pointwise nonsingular.
The corresponding reduced DAE, which is here just the original DAE,
transforms to
\[
\arraycolsep 1.2pt
\begin{array}{rl}
S\dot x_1&=\hat A_{11}(t)x_1+\hat A_{12}(t)x_2+T_2(t)^Tf(t),\\[1mm]
0&=\hat A_{12}(t)^Tx_1+\hat A_{22}(t)x_2+T_2'(t)^Tf(t).
\end{array}
\]
Solving the second equation for $x_2$ and eliminating it from the first equation,
we again obtain an inherent ODE of the form \eqref{inhself}, where
\[
J=\hat A_{11}\.-\hat A_{12}\.\hat A_{22}^{-1}\hat A_{12}^T
\]
is pointwise skew-symmetric.

Theoretically, all constructions can be performed globally.
For a numerical realization one typically uses locally smooth variants
as described in Section~\ref{sec:inherent}. The only exception is the
construction of a suitable $W$, where we are still in need of a locally
smooth variant to be used within an integration.
One possibility is given in the following, cp.~\cite{Kun20}.

We start with a reference factorization
\[
W_0^T\hat E_{11}(t_0)W_0=S
\]
which may be obtained by solving the symmetric eigenvalue problem
and then scaling the eigenvalues by congruence to $\pm1$ or by a Cholesky-like
factorization for indefinite matrices as given by~\cite{BunK77}.
We then consider the matrix function
\[
W_0^T\hat E_{11}\.W_0\.=\mat{cc}\tilde E_{11}&\tilde E_{12}\\\tilde E_{21}&\tilde E_{22}\rix,
\]
where $\tilde E_{11}^T=\tilde E_{11}\.$, $\tilde E_{12}^T=\tilde E_{21}\.$,
and $\tilde E_{22}^T=\tilde E_{22}\.$. In a sufficiently small neighborhood,
the entry $\tilde E_{11}$ is close to $I_p$, the entry $\tilde E_{22}$
is close to $-I_q$, and the entry $\tilde E_{12}$ is small in norm.
In particular, the entry $\tilde E_{11}$ is symmetric positive definite
allowing for a Cholesky factorization
\[
\tilde E_{11}\.=L_{11}\.L_{11}^T,
\]
which is a smooth process. We then get
\[
\mat{cc}L_{11}^{-1}&0\\-\tilde E_{12}^T\tilde E_{11}^{-1}&I_q\rix
\mat{cc}\tilde E_{11}\.&\tilde E_{12}\.\\\tilde E_{12}^T&\tilde E_{22}\.\rix
\mat{cc}L_{11}^{-T}&-\tilde E_{11}^{-1}\tilde E_{12}\.\\0&I_q\rix=
\mat{cc}I_p&0\\0&\tilde E_{22}\.-\tilde E_{12}^T\tilde E_{11}^{-1}\tilde E_{12}\.\rix.
\]
In a sufficiently small neighborhood, the Schur complement
$\tilde E_{22}\.-\tilde E_{12}^T\tilde E_{11}^{-1}\tilde E_{12}\.$
is symmetric negative definite allowing for a Cholesky factorization
\[
-(\tilde E_{22}\.-\tilde E_{12}^T\tilde E_{11}^{-1}\tilde E_{12})=L_{22}\.L_{22}^T,
\]
such that
\[
\mat{cc}I_p&0\\0&L_{22}^{-1}\rix
\mat{cc}I_p&0\\0&\tilde E_{22}\.-\tilde E_{12}^T\tilde E_{11}^{-1}\tilde E_{12}\.\rix
\mat{cc}I_p&0\\0&L_{22}^{-T}\rix=
\mat{cc}I_p&0\\0&-I_q\rix=S.
\]
Gathering all transformations gives the locally smooth
\[
W=W_0\mat{cc}L_{11}^{-T}&-\tilde E_{11}^{-1}\tilde E_{12}\.\\0&I_q\rix
\mat{cc}I_p&0\\0&L_{22}^{-T}\rix
\]
and all steps can be executed numerically in a smooth way using automatic differentiation.
%%%%%%%%%%%%%%%%%%%%%%%%%%%%%%%%%%%%%%%%%%%%%%%%%%%%%%%%%%%%%%%%%%%%%%%%
\section{Numerical Experiments}\label{sec:num}

The presented numerical method has been implemented using automatic differentiation in order to be able to
evaluate all needed derivatives and Jacobians.
For the determination of $\deriv{Q}{\dot Q}$ on the current interval $[t_0,t_0+h]$
one can choose between the following possibilities.
\[
\begin{array}{|l|l|}\hline
\texttt{INHERENT}&Q(t)=Q_0\\\hline
\texttt{SPIN\string_STABILIZED}&Q(t)=Q_0+(t-t_0)\dot Q_0\\\hline
\texttt{ROTATED}&Q=[\>T_2\.\>\>T_2'\>],\ \hat E_1\.T_2'=0\\\hline
\texttt{SELF\string_ADJOINT}&Q\mbox{ as described in Subsection~\ref{sub:self}}\\\hline
\texttt{SKEW\string_ADJOINT}&Q\mbox{ as described in Subsection~\ref{sub:skew}}\\\hline
\texttt{PRESCRIBED}&Q\mbox{ by user-provided routine}\\\hline
\end{array}
\]
In all cases except for the last one, one can choose between the general approach,
which includes transformation to a reduced DAE, and the simplified approach assuming that
no such transformation is necessary. Schemes based on the direct discretization
of \eqref{red1} are labelled as \texttt{DIRECT}. As numerical integration methods we use the
following discretization methods, see e.g.\ \cite{HaiW96,KunM06}.
\[
\begin{array}{|l|l|}\hline
\texttt{GAUSS-LOBATTO}&\vtop{\hsize 9.0cm\noindent\baselineskip 11pt
collocation methods for DAEs based on Gau\ss\ nodes for the differential part
and Lobatto nodes for the algebraic part, see \cite{KunMS04a}}\\\hline
\texttt{RADAU}&\vtop{\hsize 9.0cm\noindent\baselineskip 11pt
collocation methods for DAEs based on Radau nodes the simplest of which is the implicit Euler method}\\\hline
\texttt{DORMAND-PRINCE}&\vtop{\hsize 9.0cm\noindent\baselineskip 11pt
Runge-Kutta-Fehlberg methods for ODEs, see \cite{HaiNW87}}\\\hline
\texttt{GAUSS}&\vtop{\hsize 9.0cm\noindent\baselineskip 11pt
collocation methods for ODEs based on Gau\ss\ nodes}\\\hline
\end{array}
\]

\begin{experiment}\label{tst:wensch}\rm
The linear DAE
\[
\mat{cc} \delta -1 & \delta t \\ 0 & 0 \rix
\mat{c} \dot x_1 \\ \dot x_2 \rix =
\mat{cc}-\eta(\delta -1) & -\eta\delta t \\ \delta -1 & \delta t-1 \rix
\mat{c} x_1 \\ x_2 \rix + \mat{c} f_1(t) \\ f_2(t) \rix,
\]
cp.\ \cite{MaeR02}, with real parameters $\eta$ and $\delta\ne 1$
is constructed in such a way that direct discretization by the implicit Euler method
corresponds to the discretization of an inherent ODE by the explicit Euler method.
Setting $\delta=-10^5,\ \eta=0$ yields a stiff inherent ODE and we expect
stability problems when working directly with the implicit Euler method.
For our numerical experiments we have chosen $f_1,f_2$ and the initial condition
so that the solution is given by $x_1(t)=x_2(t)=\exp(-t)$.
Integration interval was $[0,1]$ and tolerance was $10^{-5}$.
The following table gives the cpu times and the number of integration steps
for the various versions of the implicit Euler method.
\[
\begin{array}{|l|c|c|}\hline
\textrm{version}&\textrm{cpu time}&\textrm{steps}\\\hline\hline
\texttt{DIRECT}&\texttt{10.31}&\texttt{97840}\\\hline
\texttt{INHERENT}&\texttt{~0.73}&\texttt{~~~10}\\\hline
\texttt{SPIN\string_STABILIZED}&\texttt{~0.60}&\texttt{~~~10}\\\hline
\texttt{ROTATED}&\texttt{~0.67}&\texttt{~~~10}\\\hline
\end{array}
\]
The stabilizing effect of discretizing an inherent ODE is obvious.
The three different versions in the choice of the inherent ODE
do not differ significantly.
\end{experiment}

\begin{experiment}\label{tst:pendulum}\rm
A mathematical model of a pendulum is given by the DAE
\[
\arraycolsep 1.2pt
\begin{array}{rl}
\dot x_3&=x_1,\\
\dot x_4&=x_2,\\
-\dot x_1&=2x_3x_5,\\
-\dot x_2&=1+2x_4x_5,\\
0&=x_3^2+x_4^2-1,
\end{array}
\]
which is known to satisfy Hypothesis~\ref{hyp:hyp} with $\mu=2$, $a=3$, and $d=2$.
The equations and unknowns are ordered in such a way that
\[
F_{\dot x}(t,x,\dot x)=\mat{ccccc}
0&0&1&0&0\\
0&0&0&1&0\\
-1&0&0&0&0\\
0&-1&0&0&0\\
0&0&0&0&0\rix,\quad
F_{x}(t,x,\dot x)=\mat{ccccc}
1&0&0&0&0\\
0&1&0&0&0\\
0&0&2x_5&0&2x_3\\
0&0&0&2x_5&2x_4\\
0&0&2x_3&2x_4&0\rix.
\]
Hence, $(F_{\dot x},F_{x})$ is self-adjoint for all arguments.
The constructions of Section~\ref{sec:symmetries}, however, are
only valid for linear DAEs and therefore not applicable.
The only valid use of an inherent ODE as presented here is by the versions
\texttt{INHERENT} and \texttt{PRESCRIBED}, since in the nonlinear case
the Jacobians do not only depend on~$t$. The following table shows
the performance of various discretization schemes when integrating
over the interval $[0,10]$ with stepsize control starting with $x(0)=(0,0,1,0,0)^T$
and using a tolerance of $10^{-5}$.
\[
\begin{array}{|l|l|c|c|c|c|}\hline
\textrm{method}&\textrm{version}&\textrm{stages}&\textrm{order}&\textrm{cpu time}&\textrm{steps}\\\hline\hline
\texttt{GAUSS-LOBATTO}&\texttt{DIRECT}&\texttt{2-3}&\texttt{4}&\texttt{~0.93}&\texttt{55}\\\hline
\texttt{RADAU}&\texttt{DIRECT}&\texttt{4}&\texttt{7}&\texttt{~0.99}&\texttt{28}\\\hline
\texttt{DORMAND-PRINCE}&\texttt{INHERENT}&\texttt{7}&\texttt{4}&\texttt{~1.33}&\texttt{47}\\\hline
\texttt{DORMAND-PRINCE}&\texttt{INHERENT}&\texttt{13}&\texttt{7}&\texttt{~1.29}&\texttt{28}\\\hline
\texttt{GAUSS}&\texttt{INHERENT}&\texttt{2}&\texttt{4}&\texttt{11.76}&\texttt{55}\\\hline
\texttt{RADAU}&\texttt{INHERENT}&\texttt{4}&\texttt{7}&\texttt{12.92}&\texttt{34}\\\hline
\end{array}
\]
In particular, we observe that we are able to solve the given problem
by explicit schemes for the chosen inherent ODE with nearly the same efficiency
as the standard direct methods.
\end{experiment}

In the following experiments we measure the geometric error in the flow~$\Phi$
with respect to a quadratic Lie group \eqref{group} by $\|\Phi^TX\Phi-X\|$,
where $\|\Delta\|=\max_{i,j=1,\ldots,n}|\Delta_{ij}|$
for $\Delta=[\Delta_{ij}]\in{\mathbb R}^{n,n}$.

\begin{experiment}\label{tst:self}\rm
The self-adjoint DAE $E(t)\dot x=A(t)x$ given by
\[
E=Q^T\hat EQ,\quad A=Q^T\hat AQ-Q^T\hat E\dot Q,
\]
where
\[
\hat E=\mat{ccc}0&1&0\\-1&0&0\\0&0&0\rix,\quad
\hat A=\mat{ccc}1&0&0\\0&1&0\\0&0&1\rix,\quad
Q=\mat{ccc}1&s&0\\s&1&s\\0&s&1\rix,
\]
with $s(t)=\frac12\sin\omega t,\ \omega=1$,
possesses a symplectic flow with respect to the first two components
of the transformed unknown $\hat x=Qx$.

The following table shows the performance and the maximal geometric error
in the flow for various discretization schemes when integrating
over the interval $[0,200\pi]$ using $1,000$ equidistant steps.
We used the simplified approach due to $\mu=0$.

\[
\begin{array}{|l|l|c|c|c|c|}\hline
\textrm{method}&\textrm{version}&\textrm{stages}&\textrm{order}&\textrm{cpu time}&\textrm{error}\\\hline\hline
\texttt{GAUSS-LOBATTO}&\texttt{DIRECT}&\texttt{2-3}&\texttt{4}&\texttt{~1.44}&\texttt{1.380e-02}\\\hline
\texttt{DORMAND-PRINCE}&\texttt{INHERENT}&\texttt{7}&\texttt{4}&\texttt{~5.36}&\texttt{2.468e-01}\\\hline
\texttt{GAUSS}&\texttt{ROTATED}&\texttt{2}&\texttt{4}&\texttt{23.68}&\texttt{7.281e-04}\\\hline
\texttt{GAUSS}&\texttt{SELF\string_ADJOINT}&\texttt{2}&\texttt{4}&\texttt{24.88}&\texttt{1.224e-07}\\\hline
\end{array}
\]
\end{experiment}

\begin{experiment}\label{tst:skew}\rm
The skew-adjoint DAE $E(t)\dot x=A(t)x$ given by
\[
E=Q^T\hat EQ,\quad A=Q^T\hat AQ-Q^T\hat E\dot Q,
\]
where
\[
\hat E=\mat{cccc}1&0&0&0\\0&1&0&0\\0&0&0&0\\0&0&0&0\rix,\quad
\hat A=\mat{cccc}0&1&0&0\\-1&0&0&0\\0&0&0&1\\0&0&-1&0\rix,\quad
Q=\mat{cccc}1&s&0&0\\s&1&s&0\\0&s&1&s\\0&0&s&1\rix,\quad
\]
with $\smash{s(t)=\frac12\sin\omega t,\ \omega=1}$,
possesses an orthogonal flow with respect to the first two components
of the transformed unknown $\hat x=Qx$.

The following table shows the performance and the maximal geometric error
in the flow for various discretization schemes when integrating
over the interval $[0,200\pi]$ using $1,000$ equidistant steps.
We used the simplified approach due to $\mu=0$.

\[
\begin{array}{|l|l|c|c|c|c|}\hline
\textrm{method}&\textrm{version}&\textrm{stages}&\textrm{order}&\textrm{cpu time}&\textrm{error}\\\hline\hline
\texttt{GAUSS-LOBATTO}&\texttt{DIRECT}&\texttt{2-3}&\texttt{4}&\texttt{~2.05}&\texttt{1.226e-01}\\\hline
\texttt{DORMAND-PRINCE}&\texttt{INHERENT}&\texttt{7}&\texttt{4}&\texttt{12.52}&\texttt{1.965e-02}\\\hline
\texttt{GAUSS}&\texttt{ROTATED}&\texttt{2}&\texttt{4}&\texttt{74.85}&\texttt{9.363e+00}\\\hline
\texttt{GAUSS}&\texttt{SKEW\string_ADJOINT}&\texttt{2}&\texttt{4}&\texttt{80.64}&\texttt{1.312e-07}\\\hline
\end{array}
\]
\end{experiment}

\begin{experiment}\label{tst:indefite}\rm
The skew-adjoint DAE $E(t)\dot x=A(t)x$ given by
\[
E=Q^T\hat EQ,\quad A=Q^T\hat AQ-Q^T\hat E\dot Q,
\]
where
\[
\hat E=\mat{ccccc}1&0&0&0&0\\0&1&0&0&0\\0&0&-1&0&0\\0&0&0&0&0\\0&0&0&0&0\rix,\quad
\hat A=\mat{ccccc}0&1&0&0&0\\-1&0&0&0&0\\0&0&0&0&0\\0&0&0&0&1\\0&0&0&-1&0\rix,\quad
Q=\mat{ccccc}1&s&0&0&0\\s&1&s&0&0\\0&s&1&s&0\\0&0&s&1&s\\0&0&0&s&1\rix,\quad
\]
with $\smash{s(t)=\frac12\sin\omega t,\ \omega=1}$,
possesses a generalized orthogonal flow in $\O(2,1)$ with respect to the first three components
of the transformed unknown $\hat x=Qx$.

The following table shows the performance and the maximal geometric error
in the flow for various discretization schemes when integrating
over the interval $[0,200\pi]$ using $1,000$ equidistant steps.
We used the simplified approach due to $\mu=0$.

\[
\begin{array}{|l|l|c|c|c|c|}\hline
\textrm{method}&\textrm{version}&\textrm{stages}&\textrm{order}&\textrm{cpu time}&\textrm{error}\\\hline\hline
\texttt{GAUSS-LOBATTO}&\texttt{DIRECT}&\texttt{2-3}&\texttt{4}&\texttt{~~3.33}&\texttt{4.548e-01}\\\hline
\texttt{DORMAND-PRINCE}&\texttt{INHERENT}&\texttt{7}&\texttt{4}&\texttt{~32.40}&\texttt{8.957e-01}\\\hline
\texttt{GAUSS}&\texttt{ROTATED}&\texttt{2}&\texttt{4}&\texttt{226.55}&\texttt{6.912e-01}\\\hline
\texttt{GAUSS}&\texttt{SKEW\string_ADJOINT}&\texttt{2}&\texttt{4}&\texttt{288.55}&\texttt{1.858e-07}\\\hline
\end{array}
\]
\end{experiment}
%%%%%%%%%%%%%%%%%%%%%%%%%%%%%%%%%%%%%%%%%%%%%%%%%%%%%%%%%%%%%%%%%%%%%%%%
\section{Conclusions}\label{sec:conclusions}
We have presented discretization methods for DAEs that are based on the integration
of an inherent ODE which is extracted from the derivative array equations
associated with the given DAE utilizing automatic differentiation.
We have shown that for this inherent ODE we can use classical discretization schemes
for the numerical integration of ODEs that cannot be used be for DAEs directly.
For self-adjoint and skew-adjoint linear time-varying DAEs
we have shown that the inherent ODE can be constructed in such a way
that it inherits these symmetry properties of the given DAE and thus
also the geometric properties of its flow.
We then have exploited this to construct geometric integration schemes
with a numerical flow that preserves these geometric properties.

\bibliographystyle{plain}
%\bibliography{km33,kmbook,km}

\end{document}